\documentclass[12pt]{amsart}
\usepackage[applemac]{inputenc}
\usepackage{bm}

\usepackage[dvipsnames,svgnames,x11names,hyperref]{xcolor}
\usepackage[hyphens]{url}
\usepackage[colorlinks=true,linkcolor=Maroon,citecolor=blue,urlcolor=blue,hypertexnames=false,linktocpage]{hyperref}
\usepackage{bookmark}
\usepackage{amsmath,thmtools,mathtools}
\mathtoolsset{showonlyrefs=true}

\usepackage{blindtext}

\newcounter{mgncount}

\usepackage{tikz}

\usepackage{fancyhdr}
\usepackage{esint}
\usepackage{enumerate}
\usepackage{pictexwd,dcpic}
\usepackage{graphicx}
\usepackage{caption}
\usepackage{graphicx}

\newtheorem*{thm-}{Theorem}
\declaretheorem[name=Theorem,numberwithin=section]{thm}
\declaretheorem[name=Remark,style=remark,sibling=thm]{rem}
\declaretheorem[name=Lemma,sibling=thm]{lemma}
\declaretheorem[name=Proposition,sibling=thm]{prop}

\declaretheorem[name=Corollary,sibling=thm]{cor}
\declaretheorem[name=Question,sibling=thm]{qu}

\declaretheorem[name=Theorem,numbered=no]{theorem}

\declaretheorem[name=Question,numbered=no]{question}

\numberwithin{equation}{section}

\newcommand{\ti}{\tilde}

\newcommand{\bbR}{\mathbb{R}}

\newcommand{\bbS}{\mathbb{S}}

\newcommand{\al}{\alpha}
\newcommand{\be}{\omega}
\newcommand{\ga}{\gamma}
\newcommand{\de}{\delta}
\newcommand{\ep}{\epsilon}

\newcommand{\la}{\lambda}

\newcommand{\si}{\sigma}

\newcommand{\De}{\Delta}

\newcommand{\cA}{\mathcal{A}}
\newcommand{\cB}{\mathcal{B}}

\newcommand{\cI}{\mathcal{I}}

\newcommand{\cL}{\mathcal{L}}
\newcommand{\cM}{\mathcal{M}}

\newcommand{\cO}{\mathcal{O}}

\renewcommand{\(}{\left(}
\renewcommand{\)}{\right)}

\newcommand{\pf}[1]{\begin{proof}#1 \end{proof}}
\newcommand{\eq}[1]{\begin{equation}\begin{alignedat}{2} #1 \end{alignedat}\end{equation}}

\newcommand{\ra}{\rightarrow}

\newcommand{\mt}{\mapsto}

\newcommand{\hp}{\hphantom}

\makeatletter
\@namedef{subjclassname@2020}{%
	\textup{2020} Mathematics Subject Classification}
\makeatother

\begin{document}
	\title[Prescribed $L_p$ curvature problem]
	{Prescribed $L_p$ curvature problem}
	\author[Y. Hu, M. N. Ivaki]{Yingxiang Hu, Mohammad N. Ivaki}
	
\maketitle
	
\begin{abstract}
In this paper, we establish the existence of smooth, origin-symmetric, strictly convex solutions to the prescribed even $L_p$ curvature problem.
\end{abstract}

\section{Introduction}
The prescribed curvature problem, asks the following question:

\begin{question}
Given a positive, smooth function $f:\bbS^n\to \bbR$, is there a closed, smooth, strictly convex hypersurface whose $k$-th elementary symmetric function of the principal curvatures, $S_k$, as a function of the
unit normal vector is $f$?
\end{question}

The case $k=n$ of this Question is the Minkowski problem, which is completely solved; see, e.g., \cite{Lew38,Nir53,CY76,Pog78,Caf90}. For the case $k<n$, in \cite{GG02}, Bo Guan and Pengfei Guan proved the following result: 

\begin{theorem}Let $\ell\geq 1$ and $f: \bbS^n\to (0,\infty)$  be a $C^{\ell,1}$-smooth, even function (i.e. $f(x)=f(-x)$ for all $x\in \bbS^n$). Then there exists an origin-symmetric, $C^{\ell+2,\alpha}$-smooth (for all $0<\alpha<1$), strictly convex hypersurface whose $S_k$ is equal to $f$.
\end{theorem}

Note that while the restriction to even (or group invariant) $f$ is unnecessary, in view of \cite[Thm. 1.4]{GG02}, dropping such a condition may lead to the non-existence of strictly convex solutions. A variation of Guan-Guan's result by dropping the evenness assumption at the cost of introducing an exponential weight factor was established by Sheng, Trudinger and Wang in \cite{STW04}. See also \cite{Che59,TW83,CNS85,Tru90,Ger96,Ger97,CNS98,GLM06,GLL12,LS13,CLW18,CX22}.

Over the past hundred years, the Minkowski problem has massively evolved, and fundamental generalizations have been put forward by many researchers. One of the most notable generalizations is known as the $L_p$-Minkowski problem, which was introduced by Lutwak in his fundamental development of Brunn-Minkowski-Firey theory; see, for example, \cite{Lut93,LO95,Lut96,LYZ00,LYZ02,LYZ04,CW06} and \cite{LR10,BLYZ12,BLYZ13,HLW16,BCD17,BIS19,BBCY19,Li19,CHLL20,HXZ21,GLW22,KM22, BG23,Mil23,IM23}. Motivated by the vast and ever-growing literature on the $L_p$-Minkowski problem, we would like to ask the following question:

\begin{question}
Given a smooth function $f:\bbS^n\to (0,\infty)$, is there a closed, smooth, strictly convex hypersurface with the support function $s$ such that for some constant $c$,
\eq{\label{Lp-CE-problem}
fs^{p-1}S_k=c\quad ?
}
\end{question}

The case $k=n$ is the well-known  $L_p$-Minkowski problem. In this paper, we prove the following theorem.
\begin{thm}\label{main thm}
Let $1<p<k+1$, $1\leq k<n$ and $\ell\geq 2$. Let $f\in C^{\ell}(\bbS^n)$ be a positive, even function. Then there exists an origin-symmetric, $C^{\ell+1,\alpha}$-smooth (for all $0<\alpha<1$), strictly convex hypersurface with the support function $s$ such that
$
fs^{p-1}S_k=1.
$	
\end{thm}

For the case $p>k+1$ and without the evenness assumption on $f$, the existence of strictly convex solutions to \eqref{Lp-CE-problem} was established in \cite{GRW15} by an elliptic method and in \cite[Thm. 3.2]{BIS21} via a flow approach. Moreover, the case $p=k+1$ was recently treated in \cite{Lee23} through a compactness argument relying on the result of \cite{GRW15}. Let us mention that, for the case $p=1$, the authors in \cite{GG02} used the special structure of $S_k$ to first obtain lower and upper bounds on the principal curvatures of the convex hypersurface, which then, in combination with an ODE argument or a result of Schneider \cite{Sch88} yields the $C^0$ estimate. Moreover, for the case $p>k+1$, we have lower and upper barriers at our disposal, and the $C^0$ estimate follows immediately. None of these approaches works for $1<p<k+1$. Here, instead, we employ a geometric lemma of Chou and Wang (see \autoref{CW-ineq}) in combination with a suitable gradient estimate in a novel way to obtain the lower and upper bounds on the support function $s$ and the $C^2$ estimate \emph{simultaneously}.

%However, the progress towards the Weingarten curvature problem is lagging and fundamental questions such as finding necessary and sufficient conditions on $f$ which ensure the existence of a strictly convex solution as well as uniqueness/non-uniqueness questions are open at large (see also \cite{Che59}).

Let us also mention that a great deal of progress has been made on the significantly important counterpart problem of \eqref{Lp-CE-problem}, where $S_k$ is replaced by $\sigma_k^{-1}$ (see the next section for the definition of $\sigma_k$), known as the $L_p$-Christoffel-Minkowski problem; see \cite{Fir67,Fir70,GM03,HMS04,GMZ06,GX18,LWW19} and \cite{Iva19,BIS23a,BIS23b}.
Our approach here for deducing the $C^2$ estimate also works for the even $L_p$-Christoffel-Minkowski problem for $1<p<k+1$ and provides a simpler argument for the proof of \cite[Thm. 1.1]{GX18}.%, which answers Question 2 in the aforementioned paper; see also \cite{Gua23}.

\section{Preparation}
Let $(\bbS^n,g,\nabla)$ denote the unit sphere of the Euclidean space $\bbR^{n+1}$ equipped with its standard round metric and Levi-Civita connection. Let $\cM$ be a closed, smooth, strictly convex hypersurface in $\bbR^{n+1}$. The Gauss map of $\mathcal{M}$, denoted by $\nu$, takes the point $y\in \mathcal{M}$ to its unique unit outward normal $x=\nu(y)\in \mathbb S^n$. The support function of $\cM$ is defined by 
\eq{
		s(x)=\langle \nu^{-1}(x),x\rangle, \quad x\in \mathbb S^n.
}
The eigenvalues of $\tau[s]:=\nabla^2s+gs$ with respect to $g$ are the principal radii of curvature $\lambda_1\leq\cdots\leq\lambda_n$; cf. \cite{Sch14}. Note that since $\tau[s]>0$, the hypersurface $\cM$ can be recovered through the map $x\mapsto s x+\nabla s$.
Let $\sigma_k=\sum_{i_1<\cdots<i_k}\lambda_{i_1}\cdots\lambda_{i_k}$ denote the $k$-th elementary symmetric function of principal radii of curvature. Then we have
\eq{
S_k=\frac{\sigma_{n-k}}{\sigma_{n}}.
}

In the following, we obtain the lower and upper bounds on the support function $s$ and the $C^2$ estimate. Our main observation is that a suitable gradient estimate (see \autoref{gradient estimate}) combined with the following geometric lemma due to Chou and Wang suffices to carry out this task all at once.

\begin{lemma}\label{CW-ineq}\cite{CW00} Let $\cM$ be a smooth, origin-symmetric, strictly convex hypersurface. Let $R=\max s$ and $r=\min s$. 
We have either $\frac{R}{r}\leq \sqrt{n+1}$ or
$
\frac{R^2}{r}\leq C_n \max_{\bbS^n} \la_n,
$
where $C_n$ is a constant depending only on $n$.
\end{lemma}

\pf{ For completeness, we give a proof of this lemma here.
Suppose $R> r\sqrt{n+1}.$
Due to the convexity, we may find two perpendicular directions, say $e_1,e_2$, such that $s(e_2)=r$ and
 \eq{
 s(e_1)> R/\sqrt{n+1}.
} 
Now project $\mathcal{M}$ to the $x_1x_2$-plane and denote the corresponding convex body by $P.$ Since $P$ is origin-symmetric, $(\pm R/\sqrt{n+1},0)$ are in the interior of $P$ and $ (0,\pm r)\in \partial P.$ 

For simplicity, $P_r:=\frac{1}{r}P$, and write $D$ for the disk of radius $1$ and centred at the origin. We have $D\subset P_r$. Let $\tilde{D}$ denote the convex hull of $D$ and $(\pm R/(r\sqrt{n+1}),0)$. Then $\partial\tilde{D}$ is the union of four tangential segments to the circle $\partial D$ and two closed arcs of $\partial D$. The four tangent lines are given by
\begin{align*}
y&=\pm\frac{1}{\sqrt{\left(\frac{R}{r\sqrt{n+1}}\right)^2-1}}\left(x\pm\frac{R}{r\sqrt{n+1}}\right).
\end{align*}
Now it is easy to verify that the rectangle \[-\frac{R}{2r\sqrt{n+1}}\leq x_1\leq \frac{R}{2r\sqrt{n+1}},\quad -\frac{1}{2}\leq x_2\leq \frac{1}{2}\]
is contained in the interior of $\tilde{D}.$ Since the ellipse
\eq{
E_0=\left\{(x_1,x_2): 4r^2(n+1)\frac{x_1^2}{R^2}+4x_2^2\leq 1\right\}
}
lies in this rectangle, $rE_0$ is contained in the interior of $P$. Therefore, for some $r/2\leq h\leq r$,
$E_1:=\{(x_1,x_2):4(n+1)\frac{x_1^2}{R^2}+ \frac{x_2^2}{h^2}\leq 1\}\subset P$
while touching $P$ at  
\eq{
\vec{a}:=\left(\frac{R}{2\sqrt{n+1}}\cos\theta,h\sin\theta\right)
}
for some $-\frac{\pi}{2}\leq \theta\neq 0 \leq \frac{\pi}{2}.$ We may assume $-\frac{\pi}{2}\leq \theta< 0$. Hence, by comparing the slope of the tangent line of $E_1$ at $\vec{a}$ with the one joining $\vec{a}$ and $(R/\sqrt{n+1},0)$, and in view of the convexity of $P$, we have
\eq{
\frac{-h\sin\theta}{\frac{R}{\sqrt{n+1}}-\frac{R}{2\sqrt{n+1}}\cos\theta}\geq -\frac{4(n+1)h^2}{R^2}\frac{\frac{R}{2\sqrt{n+1}}\cos\theta}{h\sin\theta}. 
}
This gives
$ \sin^2\theta\ge \frac{3}{4}. $
Now the claim follows from estimating the radius of curvature of $E_1$ at $\vec{a}$ from below:
\eq{
\frac{2\sqrt{n+1}\left(\frac{R^2}{4(n+1)}\sin^2\theta+h^2\cos^2\theta\right)^{\frac{3}{2}}}{hR}\geq \frac{|\sin\theta|^3R^2}{4(n+1)h}\geq \frac{1}{C_n}\frac{R^2}{r},
}
where $C_n:=32(n+1)/\sqrt{27}$.
}

\subsection{Apriori estimates}
Suppose $\cM$ with the support function $s$ is a smooth, strictly convex solution to \eqref{Lp-CE-problem} with $c=1$. Then we have
\eq{\label{main eq 2}
\frac{\si_n}{\si_{n-k}}(\tau[s])=s^{p-1}f.
}
In the remainder of this section, we will always assume $p\in (1,k+1)$ and $1\leq k<n$ and that $s$ with  $\tau[s]>0$ is an \emph{even} solution to \eqref{main eq 2}. 

From \eqref{main eq 2}, we have the following basic estimates:
\eq{\label{s1:lower-R-and-upper-r}
R^{p-k-1}\leq \frac{c_{n,k}}{\min f}, \quad r^{p-k-1}\geq \frac{c_{n,k}}{\max f}.
}
Therefore, there is a lower bound on $R$ and an upper bound on $r$. To obtain the $C^0$ and $C^2$ estimates, the following gradient estimate plays a crucial role in our argument. Unlike the one in \cite{HL13}, our proof \emph{does not} rely on any particular structure of the curvature function $S_k$.

\begin{lemma}\label{gradient estimate}
For any $0<\ga<2(p-1)/k$, there exists a constant $\beta\geq 2$, depending on $\ga$, $\,k,\,p,\,\min f$ and $\|f\|_{C^1}$ such that 
\eq{ \label{s5:weak-grad-esti}
\frac{s^2+|\nabla s|^2}{s^\ga} \leq \beta R^{2-\ga}.
}
\end{lemma}
\pf{
Let $\rho^2=s^2+|\nabla s|^2$ and $\zeta=\frac{\rho^2}{s^\ga}$ where $0<\ga<2(p-1)/k$. For simplicity, we put $\tau=\tau[s]$. We may assume $\max \zeta> R^{2-\ga}$ (i.e. $\beta>1$). Therefore, at a point $x_0$ where $\zeta$ attains its maximum we have
\eq{
(\nabla^2 s+sg)\nabla s= \frac{\ga}{2}\frac{\rho^2}{s} \nabla s.
}
That is, at $x_0$, $\nabla s\neq 0$ is an eigenvector of $\tau^{\sharp_g}$. Hence, we may find an orthonormal basis $\{e_i\}$ for $T_{x_0}\bbS^n$, such that $e_1=\frac{\nabla s}{|\nabla s|}$ and $\tau|_{x_0}$ is diagonal. In particular, $\tau_{1i}=0$ for $i=2,\ldots,n$, while 
\eq{\label{s4:critical-cond}
\tau_{11}=\frac{\ga}{2}\frac{\rho^2}{s}.
}
Moreover, at $x_0$ we have
\eq{
\zeta_{;ii}&=\frac{2}{s^\ga}(\tau_{\ell ii}s_{\ell}+\tau_{ii}^2-s\tau_{ii})-4\ga \frac{\tau_{\ell i}s_\ell s_i}{s^{\ga+1}}\\
          &\hp{=}-\frac{\ga \rho^2 (\tau_{ii}-s\de_{ii})}{s^{\ga+1}}+\ga(\ga+1)\frac{\rho^2s_i^2}{s^{\ga+2}},
 }
 and
 \eq{
\frac{\zeta_{;ii}}{\zeta}&=\frac{2}{\rho^2} \(\tau_{\ell ii}s_{\ell}+\tau_{ii}^2-s\tau_{ii}\)-4\ga \frac{\tau_{ii}s_i^2}{s\rho^2}\\
&\hp{=}-\frac{\ga(\tau_{ii}-s\de_{ii})}{s}+\ga(\ga+1)\frac{s_i^2}{s^2}.
}
Let $F=\sigma_n/\sigma_{n-k}$ and $F^{ii}=\frac{\partial F}{\partial \lambda_i}$. Note that $\nabla \tau$ is fully symmetric, and $F$ is $k$-homogeneous. Using \eqref{s4:critical-cond}, $\zeta_{;ii}\leq 0$ and $\zeta_{i}=0$ for $i=1,\ldots,n$ we have
\eq{
0 &\geq\frac{2}{\rho^2}F^{ii} \(\tau_{\ell ii}s_{\ell}+\tau_{ii}^2-s\tau_{ii}\) -4\ga F^{11}\frac{\tau_{11}s_1^2}{s\rho^2}\\
    &\hp{=}-\ga F^{ii}\frac{\tau_{ii}}{s}+\ga F^{ii}+\ga(\ga+1)F^{11}\frac{s_1^2}{s^2} \\
   & =  \frac{2}{\rho^2} \((s^{p-1}f)_{1} s_1+F^{ii}\tau_{ii}^2-ks^p f\) -2\ga^2F^{11}\frac{s_1^2}{s^2}\\
    &\hp{=}-\ga k s^{p-2}f+\ga F^{ii}+\ga(\ga+1)F^{11}\frac{s_1^2}{s^2}.
}
Therefore,
\eq{
0  &\geq 2(p-1)\frac{s^{p-2}s_1^2f}{\rho^2} +2\frac{s^{p-1}s_1f_1}{\rho^2} +\frac{2}{\rho^2}\sum_{i>1}F^{ii}\tau_{ii}^2-2k\frac{s^{p}f}{\rho^2}\\
    &\hp{=}+\(\frac{\ga^2\rho^2}{2s_1^2}-\ga(\ga-1)\)F^{11}\frac{s_1^2}{s^2}-\ga ks^{p-2}f+\ga F^{ii}\\
    &\geq 2(p-1)\frac{s^{p-2}s_1^2 f}{\rho^2}+2\frac{s^{p-1}s_1f_1}{\rho^2} -2k\frac{s^{p}f}{\rho^2}\\
    &\hp{=}+\(\frac{\ga^2\rho^2}{2s_1^2}-\ga(\ga-1)\)F^{11}\frac{s_1^2}{s^2}-\ga ks^{p-2}f.
}
For $0<\ga\leq 2$,
\eq{
\frac{\ga^2 \rho^2}{2 s_1^2}-\ga(\ga-1) \geq \frac{\ga^2}{2}-\ga(\ga-1)=\ga\left(1-\frac{\ga}{2}\right)\geq 0.
}
Now if for some $\beta\geq 2$ we had
$
\frac{\rho^2}{s^{\ga}}{}_{|_{x_0}} \geq \beta R^{2-\ga},
$
then
\eq{
s_1^2 \geq \beta R^{2-\ga}s^\ga-s^2 \geq \frac{\beta}{2}R^{2-\ga}s^\ga,
}
and hence
\eq{
0&\geq \frac{2s^{p-2}f}{\rho^2}\left((p-1) s_1^2  +s s_1( \log f)_1-ks^2 -\frac{k\ga}{2} (s^2+s_1^2)\right) \\
&=\frac{2s^{p-2}f}{\rho^2}\left( \left(p-1-\frac{k\ga}{2}\right) s_1^2  +s s_1( \log f)_1-\left(k+\frac{k\ga}{2}\right) s^2 \right) \\
&\geq\frac{2s^{p-2}f}{\rho^2}\left( \left(p-1-\frac{k\ga}{2}\right)\frac{\beta}{2} R^{2-\ga}s^\ga -c_1 \beta^{\frac{1}{2}} R^{1-\frac{\ga}{2}}s^{1+\frac{\ga}{2}}-c_2 s^2 \right)\\
&= \frac{2s^{p+\ga-2}f}{\rho^2}\left( \left(p-1-\frac{k\ga}{2}\right)\frac{\beta}{2}R^{2-\ga} -c_1 \beta^{\frac{1}{2}} R^{1-\frac{\ga}{2}}s^{1-\frac{\ga}{2}}-c_2 s^{2-\ga} \right)\\
&\geq \frac{2s^{p+\ga-2}R^{2-\ga}f}{\rho^2}\left( \left(p-1-\frac{k\ga}{2}\right)\frac{\beta}{2} -c_1 \beta^{\frac{1}{2}} -c_2\right).
}
Here, on the last three lines, we used $0<\ga<2(p-1)/k$. Moreover, the constant $c_1$ depends on $\min f$ and $\|f\|_{C^1}$, and the constant $c_2$ depends on $k$. However, we would obtain a contradiction for $\beta$ large enough. 
}

\begin{prop}\label{main prop 1}
We have
\eq{
1/C\leq s\leq C,\quad \nabla^2s+gs\geq 1/C,\quad \|s\|_{C^2}\leq C,
}
where $C>0$ is a constant depending only on $n$, $\ga$, $k$, $p$, $\min f$ and $\|f\|_{C^2}$.
\end{prop}

\pf{
Let $F:=(\sigma_n/\sigma_{n-k})^{\frac{1}{k}}.$ Note that $F^k=fs^{p-1}$. In view of the identity $\nabla^2_{ii}\sigma_1=\Delta \tau[s]_{ii}-n\tau[s]_{ii}+\sigma_1$ and concavity of $F$, there holds
\begin{align}\label{sigma1 eq}
0\leq F^{ij}g_{ij}\sigma_1\leq F^{ij}\nabla^2_{ij}\sigma_1+n(s^{p-1}f)^{\frac{1}{k}}-\Delta (s^{p-1}f)^{\frac{1}{k}}.
\end{align}
We calculate
\begin{align}\label{3}
-k\Delta (s^{p-1}f)^{\frac{1}{k}}={}&(1-p)s^{\frac{p-1}{k}-1}f^{\frac{1}{k}}\sigma_1-n(1-p)(s^{p-1}f)^{\frac{1}{k}}
\nonumber\\
&+\frac{1}{k}(1-p)(p-k-1)s^{\frac{p-1}{k}-2}	|\nabla s|^2f^{\frac{1}{k}}\\
&+2(1-p)s^{\frac{p-1}{k}-1}g(\nabla s,\nabla f^{\frac{1}{k}})-ks^{\frac{p-1}{k}}\Delta f^{\frac{1}{k}}.\nonumber
\end{align}
Thus, for $p>1$, at a point where $\si_1$ attains its maximum we have
\eq{
\sigma_1\leq c_1\left( \frac{|\nabla s|^2}{s}+R\right).
}
Due to \autoref{gradient estimate},
\eq{
\sigma_1\leq c_1\left( \frac{|\nabla s|^2}{s}+R\right)\leq c_2(s^{\ga-1}R^{2-\ga}+R)\leq c_3R\left(\frac{R}{r}\right)^{1-\ga}.
}
In view of \autoref{CW-ineq}, for some constant $C$ depending on $n,\,\ga,\,k,\,p$, $\min f$ and $\|f\|_{C^2}$ we have
\eq{
\left(\frac{R}{r}\right)^{\ga}\leq C.
}
Now, the uniform lower and upper bounds on the support function and the uniform upper bound on the principal radii of curvature follow from \eqref{s1:lower-R-and-upper-r}. Moreover, the uniform lower bound on the principal radii of curvature follows from the Newton-Maclaurin inequality
\eq{
\sigma_n=fs^{p-1}\sigma_{n-k} \geq c_{n,k} (\min f) r^{p-1}\sigma_n^{\frac{n-k}{n}}
}
and the upper bound on the principal radii of curvature.
}
\begin{rem}
In the proof of \autoref{main prop 1}, we only used the \emph{concavity} of the curvature function $(\sigma_n/\sigma_{n-k})^{\frac{1}{k}}$.
\end{rem}

In view of the Evans-Krylov and the Schauder regularity theory, as well as the apriori estimates in \autoref{main prop 1}, we obtain the following theorem.

\begin{thm}\label{main prop 2}
For any $\ell\geq 2$ and $\alpha\in (0,1)$, there exists a constant $C>0$ depending only on $n,\,k,\,p,\,\alpha$, $\min f$ and $\|f\|_{C^{\ell}}$ such that
\eq{
\|s\|_{C^{\ell+1,\alpha}}\leq C.
}
\end{thm}
\section{Existence of solutions}
It is known that the only smooth, strictly convex solution to
\eq{\label{reformulation}
\frac{\sigma_n}{\sigma_{n-k}}(\tau[s])=fs^{p-1}, \quad 1\leq k<n,\quad 1<p<k+1
}
when $f\equiv 1$ is an origin-centred sphere; see \cite[Satz 6.4]{Sim67} and \cite[Cor. 1.2]{LW23} (see also \cite{GLM18} where the uniqueness is proved for $p>1-k$). Employing this uniqueness result and the apriori estimates established in \autoref{main prop 2}, we may now closely/verbatim follow the approach in \cite{{GG02}} or \cite[Prop. 6.1]{GX18} with slight modifications to prove the existence of an origin-symmetric, smooth, strictly convex solution to \eqref{reformulation} (alternatively, one may use a result of Li \cite{Li89}; see also \cite{Li90,Li97}). For completeness, we outline the argument below.

Let $v\in C^2(\bbS^n)$ be a positive function and $q:=\frac{p-1}{k}+1\in (1,2)$. Consider the auxiliary equation
\eq{ \label{s3:auxi-eq}
F(\nabla^2 s+vg)=\(\frac{\si_n}{\si_{n-k}}\)^\frac{1}{k}(\nabla^2 s+vg)=\frac{s^{q}f}{v}.
}
We may assume $F(1,\ldots,1)=1$ for simplicity.
Moreover, define
\eq{
\cA[v]=\{s\in C^2(\bbS^n):s>0,\, \nabla^2 s+vg>0\}.
}
Using the following apriori estimates and a continuity argument, we will prove the existence of a solution to \eqref{s3:auxi-eq} in $\cA[v]$. 

\begin{prop}\label{s3:estimate-auxi-equ}Let $s\in C^4(\bbS^n) \cap \cA[v]$ be a solution to \eqref{s3:auxi-eq}.
There exists a constant $C_1>0$ depending only on $\min v$, $\min f$, $\|v\|_{C^2}$ and $\|f\|_{C^2}$ such that
\eq{
1/C_1 \leq s \leq C_1, \quad |\nabla^2 s|\leq C_1
}
\end{prop}
\begin{proof}
The lower and upper bounds on $s$ follow easily by considering the points at which $\min s$ and $\max s$ are attained. Next, we deduce the $C^2$ estimate. Note that $\nabla^2s+vg$ is positive definite; hence, we only need to obtain an upper bound on $\De s$. Assume $\De s$ attains its maximum at $x_0\in \bbS^n$. Take a local orthonormal frame around $x_0$ such that $s_{ij}(x_0)$ is diagonal (hence $\nabla^2s+gv$ and $F^{ij}$ are both diagonal). Now, at $x_0$, using $(\De s)_{;ij} \leq 0$ and concavity of $F$, we calculate
\eq{
0 &\geq F^{ii}(\De s)_{;ii} \\
&= F^{ii}(\De s_{;ii}+2\De s -2n s_{;ii})\\
&= F^{ii}\De (s_{;ii}+v)-2n F^{ii}(s_{;ii}+v)+(2\De s+2nv-\De v)\sum F^{ii}\\
&\geq  \De F-2nF+(2\De s+2nv-\De v)\sum F^{ii}.
}
We may assume that $\De s|_{x_0} \geq (-nv+\frac{1}{2}\De v)|_{x_0}$. Due to \eqref{s3:auxi-eq} and $q>1$,
\eq{
0 &\geq  \frac{f}{v}(qs^{q-1}\De s +q(q-1)s^{q-2}|\nabla s|^2)\\
&\hp{=}+2q s^{q-1}g( \nabla s, \nabla (fv^{-1}))+s^q \De(fv^{-1})-2ns^q \frac{f}{v}\\
&\geq  qs^{q-1}\frac{f}{v}\De s-2q s^{q-1}|\nabla s||\nabla (fv^{-1})|+s^{q}\De(fv^{-1})-2ns^q \frac{f}{v}.
}
It follows that
\eq{
\De s \leq c_1 s+c_2|\nabla s|,
}
where $c_1$, $c_2$ are positive constants depending on $\min v$, $\min f$, $\|v\|_{C^2}$ and $\|f\|_{C^2}$, as $q\in (1,2)$. Moreover, by \cite[Lem. 3.1]{GG02}, for any function $s\in C^2(\bbS^n)$, we have
\eq{
\|s\|_{C^1}^{2}\leq 4\|s\|_{C^0}\|s\|_{C^2}.
}
Therefore,  $\De s \leq C_1$, where $C_1$ is a positive constant depending on $\min v$, $\min f$, $\|v\|_{C^2}$ and $\|f\|_{C^2}$.
\end{proof}

\begin{cor}\label{s3:cor-eigenvalue-esti}
    Let $s\in C^4(\bbS^n)\cap \cA[v]$ be a solution to \eqref{s3:auxi-eq}. Then there exists a constant $C_2>0$ depending only on $\min v$, $\min f$, $\|v\|_{C^2}$ and $\|f\|_{C^2}$ such that 
    \eq{
  C_2^{-1} g \leq \nabla^2 s+v g \leq C_2 g.
    }
\end{cor}
\begin{proof}
By the Newton-MacLaurin inequality, we have 
    \eq{
    \(\frac{\si_n}{\si_{n-k}}\)^\frac{1}{k}(\nabla^2 s+vg) \leq c_{n,k} \si_n^\frac{1}{n}(\nabla^2 s+vg).
    }
In view of Proposition \ref{s3:estimate-auxi-equ}, the left-hand side is bounded from below while we know the eigenvalues of $\nabla^2 s+vg$ are bounded from above. Hence, the eigenvalues of $\nabla^2 s+vg$ are bounded from below.
\end{proof}

\begin{thm}\label{s3:thm-C2-estimate}
    Assume $v,f\in C^4(\bbS^n)$, $v>0$ and $f>0$. Then there exists a unique solution $s\in C^{5,\al}(\bbS^n) \cap \cA[v]$ to \eqref{s3:auxi-eq}, where $0<\al<1$. Moreover, there exists $C>0$ depending only on $n$ and $\al$, such that
    \eq{
    \|s\|_{C^{5,\al}} \leq C.
    }
\end{thm}
\begin{proof} Let $v^t:=tv+(1-t)$ and $f^t=tf+(1-t)$.
    We show that for each $t\in [0,1]$, the equation
    \eq{\label{s3:degree-eq}
    F(\nabla^2 s+v^tg)=\frac{s^q}{v^t}f^t
    }
   admits a unique solution $s^t$ in $\cA[v^t]$. We proceed to prove the existence part. Once the existence is established, the uniqueness follows from the comparison theorem; see, e.g. \cite[Thm. 17.1]{GT83}). Let us define
\begin{align*}
    \cI=\left\{ \text{$a\in [0,1]$: \eqref{s3:degree-eq} is solvable in $C^{5,\al}(\bbS^n)\cap \cA[v^t]$ for all $t\in[0,a]$}\right\}.
\end{align*}
Note that $\cI\neq \emptyset$;
 $s^0=1$ is the solution of \eqref{s3:degree-eq}, and thus $0\in\cI$. Now, we show that $\cI$ is closed in $[0,1]$. Let $t\in \cI$ and $s^t\in C^{5,\al}(\bbS^n)\cap \cA[v^t]$ be the solution to \eqref{s3:degree-eq}. In light of Corollary \ref{s3:cor-eigenvalue-esti} and Proposition \ref{s3:estimate-auxi-equ},
    \eq{
    \|s^t\|_{C^2} \leq C,
    }
    where $C$ is independent of $t$. Moreover, the Evans-Krylov regularity theory, in combination with the Schauder estimate, implies that, for some  $C$ is independent of $t$:
    \eq{
    \|s^t\|_{C^{5,\al}} \leq C.
    }
Thus, $\cI$ is closed in $[0,1]$. 

Next, we show that $\cI$ is open in $[0,1]$. Note that the linearized operator of the mapping
    \eq{
s \mt F(\nabla^2 s+v^t g)-\frac{s^{q}}{v^t}f^t
    }
    at $s^t$ is given by
    \eq{
   \cL_{t} \eta=F^{ij}(\nabla^2 s^t+v^tg)\eta_{;ij}-q(s^t)^{q-1}\frac{f^t}{v^t}\eta, \quad \eta\in C^2(\bbS^n).
    }
By the maximum principle,
$
    \cL_t: C^{5,\al}(\bbS^n) \ra C^{3,\al}(\bbS^n)
$
    is one-to-one. Thus, $\cL_t$ is invertible if and only if its index is zero; i.e. $\operatorname{ind}(\cL_t)=0$. To show that $\operatorname{ind}(\cL_t)=0$, due to the homotopy invariance of the index, it suffices to show that  $\operatorname{ind}(\cL_0)=0$. For some $c>0$, the linearized operator $\cL_0$ is given by
    \eq{
    \cL_0 \eta=c\De\eta-q \eta, \quad \eta\in C^{2}(\bbS^n).
    }
Hence, it is invertible from $C^{5,\al}(\bbS^n)$ onto $C^{3,\al}(\bbS^n)$ and $\operatorname{ind}(\cL_0)=0$. By the implicit function theorem, $\cI$ is open in $[0,1]$ and thus $\cI=[0,1]$.    
\end{proof}

\begin{proof}[Proof of \autoref{main thm}]
Let us consider the Banach space 
\eq{
\cB:=\{w\in C^4(\bbS^n) : w(x)=w(-x) \quad \forall x\in \bbS^n\}.
}
Recall that $f\in \cB$ and $f>0$. For $w\in \cB$ and $0\leq t\leq 1$, we write $v=e^w$ and let $s^t$ denote the unique solution of \eqref{s3:auxi-eq} in $\cA[v]$ with $f^t=t f+(1-t)$. Due to the uniqueness,  $s^t$ is even; hence, $\log s^t\in \cB$. Let $I: \cB\to \cB$ be the identity map and define $\phi_t:\cB \ra \cB$ by
\eq{ \label{s3:T-operator}
\phi_t(w)=\log s^t.
}
Then, due to Theorem \ref{s3:thm-C2-estimate}, $\phi_t$ is compact. Consider the open set
\eq{
\cO:=\{w\in \cB:~\|w\|_{C^5}<R\}.
}
By \autoref{main prop 2}, $\phi_t$ does not have any fixed points on the boundary of $\cO$, provided $R$ is sufficiently large. Therefore, $\deg(I-\phi_t,\cO,0)$ is well-defined, provided $R$ is sufficiently large.

Let $w$ be a fixed point of $\phi_0$; that is, $w$ satisfies $\phi_0w=w$. So $s=e^w$ is a solution to the equation
\eq{\label{s3:isotropic-Weingarten-eq}
\(\frac{\si_n}{\si_{n-k}}\)^\frac{1}{k}(\tau[s])=s^{q-1}, \quad q=\frac{p-1}{k}+1.
}
By the uniqueness result \cite[Satz 6.4]{Sim67},  $e^w=1$ and $w=0$. Hence,
\eq{
\deg(I-\phi_0,\cO,0)=\deg(I-\phi_0,B_\de(0),0)\quad \text{for any small}~ \de>0.
}
Set $\ti{\phi}_0 v=e^{\phi_0(\log v)}$. Then
\eq{
\deg(I-\phi_0,B_\de(0),0)=\deg(I-\ti \phi_0,B_\de(1),0).
}
Let $\ti{\phi}'_{0,v}$ denote the linearized operator of $\ti \phi_0$ at $v$. Our next goal is to show that $I-\ti{\phi}'_{0,1}$ is invertible. Note that
\eq{
F(\nabla^2(\ti{\phi}_0v)+vg)&=\frac{(\ti{\phi}_0 v)^q}{v},\\
F^{ij}(\nabla^2(\ti \phi_0v)+vg)((\ti \phi'_{0,v}\eta)_{;ij}+\eta \de_{ij})&=\frac{q(\ti \phi_0v)^{q-1}\ti \phi'_{0,v}\eta}{v}-\frac{(\ti \phi_0v)^{q}\eta}{v^2}.
}
In particular, for $v=1$, we have $\ti \phi_0v=e^{\phi_0(0)}=1$, $(\ti \phi_0v)_{ij}+v\de_{ij}=\de_{ij}$ and $F^{ij}=c\de_{ij}$ for some constant $c>0$. Therefore,
\eq{\label{eig-equation}
(q-c\De)(\ti{\phi}'_{0,1}\eta)=(cn+1)\eta.
}
From this identity and that $q>1$ we obtain
$
\ti{\phi}'_{0,1}=(n+\frac{1}{c})(\frac{q}{c}-\De)^{-1}.
$
If $(I-\ti \phi'_{0,1})\eta=0$, then $\eta$ satisfies $\De \eta=(-n+\frac{q-1}{c})\eta$, and due to $q>1$ we must have $\eta=0$. Therefore, $I-\ti \phi'_{0,1}$ is injective, and by the Fredholm alternative, it is invertible. Moreover, in view of \cite[Thm. 2.8.1]{Nir01},
\eq{
\deg(I-\ti \phi_0,B_\de(1),0)=(-1)^\beta,
}
where $\beta$ is the number of eigenvalues of $\ti \phi'_{0,1}$ greater than one. To compute $\beta$, let $\eta\in C^{2}(\bbS^n)$ satisfy
$
\ti{\phi}'_{0,1} \eta=\lambda \eta,
$
with some $\lambda>1$. Then \eqref{eig-equation} gives
\eq{
\De \eta=\left(\frac{q}{c}-\frac{n+\frac{1}{c}}{\lambda}\right)\eta.
}
Since $\frac{q}{c}-\frac{n+\frac{1}{c}}{\lambda}>\frac{q}{c}-n-\frac{1}{c}>-n$, and $0$ is the only eigenvalue of $\De$ greater than $-n$, we must have $\lambda=\frac{cn+1}{q}$. That is, $\beta=1$ and 
\eq{
\deg(I-\phi_t,\cO,0)=-1.
}
Hence, for each $0\leq t\leq 1$, $\phi_t$ has a fixed point in $\mathcal{B}$. In particular, for $t=1$, there exists a $C^4$-smooth, strictly convex solution $s$ to \eqref{main eq 2}, settling the existence part of our main theorem. The regularity of $s$ is a consequence of \autoref{main prop 2}.
\end{proof}

\section*{Acknowledgment}
We would like to thank the referee for their comments. The work of the first author was supported by the National Key Research and Development Program of China 2021YFA1001800, the National Natural Science Foundation of China 12101027, and the Fundamental Research Funds for the Central Universities. Both authors were supported by the Austrian Science Fund (FWF) under Project P36545.
 
\providecommand{\bysame}{\leavevmode\hbox to3em{\hrulefill}\thinspace}

	\vspace{10mm}
	\textsc{School of Mathematical Sciences, Beihang University, Beijing 100191, China,}
	\email{\href{mailto:huyingxiang@buaa.edu.cn}{huyingxiang@buaa.edu.cn}}
	
	\vspace{5mm}
	\textsc{Institut f\"{u}r Diskrete Mathematik und Geometrie,\\ Technische Universit\"{a}t Wien, Wiedner Hauptstra\ss e 8-10,\\ 1040 Wien, Austria,}
	\email{\href{mailto:yingxiang.hu@tuwien.ac.at}{yingxiang.hu@tuwien.ac.at}}
	
	\vspace{5mm}
\textsc{Institut f\"{u}r Diskrete Mathematik und Geometrie,\\ Technische Universit\"{a}t Wien, Wiedner Hauptstra{\ss}e 8-10,\\ 1040 Wien, Austria,} \email{\href{mailto:mohammad.ivaki@tuwien.ac.at}{mohammad.ivaki@tuwien.ac.at}}

\end{document}